\documentclass[12pt]{article}

\usepackage{pgfplots}
\usepackage{times}
\usepackage{amsmath,amsfonts, amstext,amssymb,amsbsy,amsopn,amsthm}
\usepackage{dsfont}
\usepackage{esint}
\usepackage{graphicx}   
\usepackage{hyperref}
\usepackage[all]{xy}
\usepackage{color}
\usepackage{tikz}
\usepackage{extarrows}
\usepackage{lineno,hyperref}
\usetikzlibrary{intersections}
\newtheorem{theorem}{Theorem}[section]

\newtheorem{definition}{Definition}[section]

\newtheorem{lemma}[definition]{Lemma}
\newtheorem{proposition}[definition]{Proposition}

\theoremstyle{remark}
\newtheorem{remark}[definition]{Remark}
\numberwithin{equation}{section}

\newcommand{\norm}[1]{\lVert#1\rVert}

\newcommand{\R}{\mathbb{R}}
\newcommand{\cL}{\mathcal{L}}

\newcommand{\flap}{\mbox{$(-\triangle)^s$}}

\newcommand{\be}{\begin{equation}}
\newcommand{\ee}{\end{equation}}
\newcommand{\bee}{\begin{equation*}}
\newcommand{\eee}{\end{equation*}}

\title{  Boundary regularity  and a priori estimates  for  fractional equations on unbounded domains}

\author{Yahong Guo,  Congming Li, and Yugao Ouyang}

\begin{document}

\maketitle

\begin{abstract}

    In this paper, we study the boundary H\"older regularity for solutions to the fractional Dirichlet problem in unbounded domains with boundary
\begin{equation*}
        \begin{cases}
            \flap u(x) = g(x),&\text{in } \Omega,\\
            u(x)=0, &\text{in }  \Omega^c.\\
        \end{cases}
    \end{equation*}  
   Existing  results 
    rely on the global $L^{\infty}$ norm of solutions to control their boundary $C^s$ norm, which  is insufficient for  blow-up and rescaling analysis  to obtain a priori estimates  in unbounded domains. To overcome this limitation,  we first derive a local version of boundary H\"older regularity for nonnegative solutions in which we replace the global  $L^{\infty}$ norm by only a local $L^{\infty}$ norm. Then as an important application, we establish a priori estimates for nonnegative solutions to a family of nonlinear equations on unbounded domains with boundaries.

 \end{abstract}
 \bigskip

\bigskip

\medskip

\section{Introduction}

 In this paper, we consider the following fractional problem in  unbounded domain $\Omega$
 \begin{equation}\label{main1}
        \begin{cases}
            \flap u(x) = g(x),&\text{in } \Omega,\\
            u(x)=0, &\text{in }  \Omega^c.\\
        \end{cases}
    \end{equation}
First, we derive a boundary regularity estimate for  solutions   
in terms of only the local $L^{\infty}$ norm of $u$ and $g$. 
Then as an important application, we obtain  a priori estimates for solutions to nonlinear equations on unbounded domains with  boundaries \begin{equation}\label{main2}
\begin{cases}
    \flap u (x) = f(x,u)  , & x\in \Omega,\\
    u (x)  = 0 , &x\in \Omega ^c.
\end{cases}
\end{equation} 
Here, $(-\Delta)^s$  denotes the fractional Laplacian,  defined as
\begin{eqnarray}\label{eq1-1}
\begin{aligned}
(-\Delta)^s u(x)&=C_{n, s}P.V. \int_{\mathbb{R}^n}\frac{u(x)-u(y)}{|x-y|^{n+2s}}dy\\
&=C_{n, s}\lim_{\varepsilon \to 0}\int_{\mathbb{R}^n\backslash B_\varepsilon(0)}\frac{u(x)-u(y)}{|x-y|^{n+2s}}dy,
\end{aligned}
\end{eqnarray}
where $P.V.$
stands for the Cauchy principal value, and $C_{n, s}$ is a
dimensional constant.
In order that the integral on the right hand side of \eqref{eq1-1} is well defined in the classical sense, we require that $u\in C_{loc}^{1, 1}(\Omega)\cap \mathcal{L}_{2s}$, where
$$
\mathcal{L}_{2s}=\left\{u\in L_{loc}^1(\R^n) \,\Big| \int_{\mathbb{R}^n}\frac{|u(x)|}{1+|x|^{n+2s}}dx<\infty \right\}
$$
endowed with norm 

\[\norm{u}_{{\cL}_{2s}}:=\int_{\mathbb{R}^n}\frac{|u(x)|}{1+|x|^{n+2s}}dx.\]

The fractional Laplacian is a nonlocal integro-differential operator that effectively models memory effects and long-range diffusion phenomena (Chaves \cite{Chave1998}; Metzler and Klafter \cite{MK2000}). Specific applications include random walks with memory, Lévy flights \cite{MS1984, SZK1993, N1986}, and kinetic theories of systems exhibiting chaotic \cite{SZ1997, Z2002} and pseudo-chaotic \cite{ZE2001} dynamics. Beyond these, the operator finds significant applications across diverse domains including probability theory and finance, image processing and computer vision, mathematical biology and ecology, as well as geophysics and climate modeling.

 Nonlocal operators, particularly the fractional Laplacian, have attracted considerable attention in the mathematical community owing to their broad applications. For foundational developments in this field, we refer to the seminal work of Caffarelli and Silvestre \cite{CS07}, along with subsequent contributions by  Stinga \cite{Stinga2019}, Dipierro et al. \cite{Dipierro2019}, and the comprehensive monograph by Bucur and Valdinoci \cite{BV2016}. Regarding qualitative properties of solutions, extensive results can be found in \cite{CDQ,CDG,CG,CGLO,CHL17,DQ, FL,FLS}, with specific advances in a priori estimates documented in \cite{CGL,Chen2016directblowup,ouyang2023}.
 
    The regularity properties of solutions to fractional order equations have also been extensively studied. The foundation was laid by Silvestre \cite{Silvestre2007Regularity}, who established Schauder and H\"older estimates for entire solutions. Then the interior type estimates are studied in \cite{Chen2010frac}.  Later, Li and Wu \cite{Li2021Pointwise} established interior pointwise regularity estimates for fractional equations, a result analogous in spirit to Caffarelli's seminal work \cite{Caffarelli1989Ann} on pointwise estimates for viscosity solutions of fully nonlinear elliptic equations.
    For boundary regularity, Ros-Oton and Serra \cite{Rosoton2014Dirichlet} proved optimal regularity for the Dirichlet problem in domains satisfying an exterior ball condition (with zero exterior data), while their follow-up work \cite{Rosoton2014extremal} treated general exterior data in $C^{1,1}$ domains.
    The regularity theory has been further developed from different perspectives: Servadei and Valdinoci \cite{Servadei2014weak} investigated both weak and viscosity solutions, while Caffarelli-Stinga \cite{CS2016} and Stinga-Torrea \cite{ST2010} provided additional important regularity criteria.
    For more results concerning the nonlocal elliptic equations, please see \cite{Fall2019regulariy,Fall2020reg,Dyda2020reg}, while for fractional parabolic equations, we refer the readers to \cite{CGL,Chen2010Heat,Bogdan2010heat,Fernandezreal2016boundary,Stinga2017reg,Kassmann2024parabolic}.  

In particular, in \cite{chen2025refinedregularitynonlocalelliptic}, the authors establish interior H\"older and Schauder regularity estimates for fractional equation \eqref{main1}. This makes it possible to derive a priori estimate for nonnegative solutions of \eqref{main2} when $\Omega=\R^n$, an  unbounded domain {\em without boundary}.

Then for general unbounded domains  with  boundaries, can one derive a priori estimates for the solutions? 

This requires boundary H\"older estimates in terms of only the local $L^{\infty}$ norm of the solutions.

However, in the existing results, such as the ones in \cite{BDGQ2018} and \cite{Barrios2019note}, in order to obtain the $C^s$ boundary H\"older continuity of the solutions for \eqref{main1} when $\Omega=\R^n_+$, the authors  require that solution $u$ is globally bounded. Their proof mainly based on the proposition 1.1 in \cite{Rosoton2014Dirichlet}, in which  an optimal boundary H\"older estimate in terms of $\norm{g}_{L^\infty(\Omega)}$ is established for  a {\em bounded domian} $\Omega$ with zero exterior Dirichlet condition. It is well known that in this situation, one can construst a supersolution  to show that the global $L^{\infty}$ norm of $u$ is bounded by $\norm{g}_{L^\infty(\Omega)}$. When $\Omega$ is an {\em unbounded domain}, it is evident   from their proof that  in order to obtain local boundary H\"older continuity, the solution $u$ is required to be globally bounded, as indicated by \cite{BDGQ2018} and \cite{Barrios2019note}. For general unbounded domains, Dipierro, Soave and Valdinoci in \cite{DSV17} obtain $C^{\alpha}$ boundary regularity for some $0<\alpha<s$ under the condition that the solution $u$ of \eqref{main1} be globally bounded. 

With these existing boundary  regularity estimates, it is inadequate to carry out blow up and rescaling arguments on unbounded domains aimed at obtaining a priori estimates, because the rescaled functions may be globally unbounded. 

One of our main result is to establish the local version of the boundary regularity for nonnegative solutions,  in which  we use {\em only the local $L^{\infty}$  norm of $u$ instead of the global ones} to control its local H\"older norm up to the boundary.
\begin{theorem}\label{1bdry C^s}
    Suppose $\Omega$ is a locally $C^{1,1}$ domain, $0<s<1$, $g\in L^\infty(\Omega\cap B_4)$ and $u$ is a nonnegative classical solution of 
    \begin{equation}
        \begin{cases}
            \flap u = g,&\text{in } \Omega\cap B_4,\\
            u=0, &\text{in } B_4\backslash \Omega.\\
        \end{cases}
    \end{equation}
    Then $u\in C^{s}(\Omega\cap B_{1/2})$ and 
    \[
    \norm{u}_{C^s(\Omega\cap B_{1/2})}\le C(\norm{g} _{L^\infty(\Omega\cap B_4)} + \norm{u} _{L^\infty(\Omega\cap B_4)} ),
    \]
    where $C$ depends on $n,s$ and $C^{1,1}_{loc}$ norm of $\partial\Omega$, and $B_r$ denotes a ball centered at any fixed point on the boundary $\partial\Omega$ with radius $r$.
\end{theorem}
The idea of the proof is that we divide a given solution into two parts: the potential part and the harmonic part. The regularity for the potential part is obtained by Proposition \ref{otonserra}. For the harmonic part $h$, we rewrite it in terms of the Poisson representation formula in balls. Using this explicit expression, we  first carry out a detailed analysis to   derive an $\alpha$ power order decay near $\partial\Omega$,
\begin{equation*}
        |h(x)|\le C\{\norm{u}_{L^\infty(\Omega \cap B_4)}+\norm{u}_{\cL_{2s}}\}\operatorname{dist}(x,\partial \Omega)^{\alpha },
    \end{equation*}
    where $\alpha = \min\{s,1-s\}$. Combining this decay estimate and the interior regularity result (Theorem \ref{Holder thm}), we derive the $C^\alpha$ boundary regularity. Then by an iteration process, we increase the power $\alpha$ successively until it reaches the desired power $s$.

As an important application, we establish a priori estimate for the solutions  to \eqref{main2}.  We assume that $\Omega \subset \R ^n$ is an unbounded domain with uniformly $C^{1,1}$ boundary, and $f$ satisfies the following condition:
\begin{itemize}
    \item $f(x,t):\Omega \times [0,\infty)\to\R$ is uniformly H\"older continuous with respect to $x$ and continuous with respect to $t$.
\end{itemize}
\begin{theorem}\label{1A1}
Assume $1<p<\frac{n+2s}{n-2s}$, $f$ satisfies 
\begin{equation}\label{growth f}
    f(x,t)\le C_0(1+t^p)
\end{equation}
uniformly for all $x$ in $\Omega$, and 
\begin{equation}\label{lim f}
    \lim _{t\to\infty}\frac{f(x,t)}{t^p} = K(x),
\end{equation}
    where $K(x)\in (0,\infty)$ is uniformly continuous and $K(\infty):=\lim _{|x|\to\infty }K(x) \in (0,\infty)$. 
Then there exists a constant $C$, such that 
\begin{equation}\label{bd1}
u(x)\le C,\quad\forall \, x\in \Omega
\end{equation}
holds for all nonnegative solutions $u$ of \eqref{main2}.
\end{theorem}

\begin{remark}
    The assumptions on $f$ are essential for the a priori estimate. When $f=0$ and $\Omega = B_1$, Abatangelo, Jarohs and Salda\~na \cite{Abatangelo2018Green} constructed nontrivial nonnegative solutions to \eqref{main2} that are unbounded, see also \cite{Li2023unique}.
\end{remark}

During the process of blow-up and rescaling, it can be shown that a sequence of the rescaled functions (still denoted by $u_i$) converges to a limiting function $u$ in the sense of $C^{2s+\beta}_{loc}$. In the case of a local operator, say $-\Delta$, we already have 
\[   \lim _{i\to\infty}(-\Delta) u_i(x) = (-\Delta) u(x).\]
However, this is not the case for nonlocal operators, as shown in \cite[Theorem 1.1]{Du2023blowup} by Du-Jin-Xiong-Yang, in which they work on the whole space $\R^n$. In order to establish a priori estimate on general  unbounded domains with boundaries, we generalize  their convergence result as follows.
\begin{theorem}\label{conv}
    Assume $n\ge 1$, $s\in (0,1)$, $\beta\in (0,1)$. $\Omega$ is an unbounded domain with boundary. Suppose nonnegative functions $\{u_i\}\subset \mathcal{L}_{2s}\cap C^{2s+\beta}_{loc}(\Omega)$ vanishes outside $\Omega$, and $u\in \mathcal{L}_{2s}$ vanishes outside $\Omega$. If $\{u_i\}$
    converges in $C^{2s+\beta}_{loc}(\Omega)$ to a function $u\in \mathcal{L}_{2s}$ , and $\{\flap u_i \}$ converges pointwisely in $\Omega$, then there exists a constant $b\ge 0 $ such that
    \begin{equation}
        \lim _{i\to\infty}\flap u_i(x) = \flap u(x) -b.
    \end{equation}
\end{theorem}

 In Section 2, we establish local boundary H\"older regularity  for \eqref{main1} and prove Theorem \ref{1bdry C^s}. In Section 3, we derive  a priori estimates for \eqref{main2} and validate Theorem \ref{1A1}. Finally, the proof of Theorem \ref{conv} is included in the Appendix.

\section{Local Boundary H\"older Regularity}
\,\,\,\,\,\, \,\,In this section, we prove Theorem \ref{1bdry C^s} and thus establish a local version of the boundary H\"older regularity for  fractional equations. 

To begin with, let us recall a classical result in
\cite{Rosoton2014Dirichlet}, in which the authors establish an  optimal boundary regularity for the fractional equation with zero exterior Dirichlet condition.
\begin{proposition}[\cite{Rosoton2014Dirichlet}]\label{otonserra}
    Suppose $\Omega$ is a bounded Lipschitz domain with exterior ball condition, $g\in L^\infty(\Omega)$, $s\in (0,1)$, $u$ is a classical solution of the following Dirichlet problem:
    \begin{equation}\label{eq1}
        \begin{cases}
            \flap u = g,&\text{in } \Omega,\\
            u=0, &\text{in } \Omega ^c.\\
        \end{cases}
    \end{equation}
    Then $u\in C^{s}(\R^n)$ and 
\begin{equation}\label{hd}
    \norm{u}_{C^s(\R^n)}\le C\norm{g}_{L^\infty(\Omega)},
\end{equation}
where $C=C(n,s)$ is a positive constant.
\end{proposition}
Notice that here $\Omega$ is a {\em bounded domian}. In this situation,  a supersolution  and a subsolution can be constructed  to show that the global $L^{\infty}$ norm of $u$ is bounded by $\norm{g}_{L^\infty(\Omega)}$ as the following. Let 
\[\psi(x):=(1-|x|^2)_+^s.\]
It is well known that 
\[\flap \psi (x)=a>0.\]
Choose $R$ sufficiently large, so that $\Omega\subset B_R(0).$ Denote
\[\bar{u}(x):=\norm{g}_{L^{\infty}(\Omega)}\frac{R^{2s}}{a}\psi_R(x), ~\mbox{with}~ \psi_R(x):={\psi(\frac{x}{R})}.\]
Then it can be easily verified that $\bar{u}$ and $-\bar{u}$ are super-solution and sub-solution of \eqref{eq1}, respectively. Therefore, 
\begin{equation*}
    \norm{u}_{L^{\infty}(\Omega)}\le C\norm{g}_{L^\infty(\Omega)},
\end{equation*}
From the above construction, one can see that in order to obtain \eqref{hd}, it is necessary that $\Omega$ be bounded.
When $\Omega$ is an {\em unbounded domain},  it is evident that the solution $u$ is required to be globally bounded. This requirement can not be fulfilled  in the process of  employing  the blow-up and rescaling argument to obtain a priori estimates for solutions to a corresponding family of nonlinear fractional equations on {\em unbounded domains with boundaries}.

This motivates  us to  establish a local version of the boundary regularity, in which, instead of the global one, only a local ${L^{\infty}}$ norm of the solution is involved.
\begin{theorem}\label{bdry C^s}
    Suppose $\Omega$ is a unbounded domain  with locally $C^{1,1}$ boundary, $0<s<1$, $g\in L^\infty(\Omega\cap B_4)$ and $u$ is a nonnegative classical solution of
    \begin{equation}
        \begin{cases}
            \flap u = g,&\text{in } \Omega\cap B_4,\\
            u=0, &\text{in } B_4\backslash \Omega.\\
        \end{cases}
    \end{equation}
    Then $u\in C^{s}(\Omega\cap B_{1/2})$ and 
    \[
    \norm{u}_{C^s(\Omega\cap B_{1/2})}\le C(\norm{g} _{L^\infty(\Omega\cap B_4)} + \norm{u} _{L^\infty(\Omega\cap B_4)} ),
    \]
    where $C$ depends on $n,s$ and $C^{1,1}_{loc}$ norm of $\partial\Omega$, and $B_r$ denotes a ball centered at any fixed point on the boundary $\partial\Omega$ with radius $r$.
\end{theorem}
The idea of the proof is that we divide the solution into two parts: the potential part and the harmonic part. The regularity for the potential part is obtained by Theorem \ref{otonserra}. For the harmonic part, we firstly derived an $\alpha$ power order decay near $\partial\Omega$, here $\alpha$ is a certain positive number which may be less than $s$. Combining the decay estimate and the interior regularity result (Theorem \ref{Holder thm}), we derive the $C^\alpha$ boundary regularity. Finally we refine the boundary regularity by an iteration progress. 

In order to analyze the harmonic part of the solution, we rewrite it by using the Poisson representation formula in balls. Using explicit expressions, we can carry out a detailed analysis to derive decay estimate.

\begin{proof}
     Without loss of generality, we may assume that $0\in\partial\Omega$ and the ball $B_r$ is centered at the origin. Since $\partial \Omega$ is locally $C^{1,1}$, for any $z\in \partial\Omega\cap B_2$, we can always find a ball $B\subset \Omega \cap B_4$ tangent to $\partial \Omega$ at $z$ with a universal radius independent of $z$. Without loss of generality, we assume such balls have the same radius 1 for any $z\in \partial\Omega\cap B_2$.
     
    Fix sufficiently small $\varepsilon \in (0,1/2)$. 
    For any $x\in (\Omega \backslash \Omega _\varepsilon) \cap B_1$ with 
    \[
    \Omega _\varepsilon := \{x\in\Omega : \operatorname{dist}(x,\partial\Omega)\ge \varepsilon\}.
    \]
    Since $\partial \Omega$ is locally $C^{1,1}$, there exists a unique $x_0\in \partial \Omega\cap B_2$ such that 
    \[
  d:= \operatorname{dist}(x,\partial \Omega) = |x-x_0|.
    \]
   { Denote $x=x_0 + d\nu$, where $\nu = \frac{x-x_0}{|x-x_0|}$, then $d<\varepsilon$. }

    Suppose $B_1(a)\subset  \Omega \cap B_4$ tangent to 
    $\partial \Omega $ at $x_0$, here $a = x_0+\nu$. For clarity, the following figure \ref{fig:geometry} illustrates the relative positions of $x_0$, $x$, and the tangent ball $B_1(a)$.

\begin{figure}[htbp]
\centering
\begin{tikzpicture}[scale=0.8]
    \filldraw (0,0) circle (1.5pt) node[below] {$o$};
    \draw[black, thick, smooth, domain=-5:5, samples=100] 
        plot (\x, {1/20*\x*\x});
    \draw[black, thick, dashed, smooth, domain=-5:5, samples=100] 
        plot (\x, {1/20*\x*\x+0.4});
    \draw[blue, thick] (0.3,1.02) circle (1);
    \node at ([shift=(45:1.47)] 0.3,1.02) {${\color{blue}B_1(a)}$};
    \filldraw[blue] (0.3,1.02) circle (1.5pt) node[right]{$a$};
     \filldraw[black] (0.35,0.02) circle (1.5pt) node[below right] {$x_0$} ;
    \filldraw[black]  (0.342,0.18) circle (1.5pt) circle (1.5pt) node[right] {$x$} ;
    \draw[-stealth, black, thick] (0.35,0.02) -- (0.3,1.02);
    \begin{scope}
        \clip (-5,0) rectangle (5,6);
        \clip (-5,6) -- (5,6) -- (5,0) -- 
              plot[domain=5:-5, smooth, samples=100] (\x, {1/20*\x*\x}) -- 
              cycle;
        \draw[black, thick] (0,0) circle (1);
        \draw[black, thick] (0,0) circle (4);
    \end{scope}
    \node at ([shift=(130:3.1)] 0,0) {$\Omega\cap B_4$};
    \node at ([shift=(192:1.2)] 0,0) {$\Omega\cap B_1$};
    \node at ([shift=(9:4.8)] 0,0) {$\partial\Omega$};
    \node at ([shift=(165:4.8)] 0,0) {$\varepsilon$};
\end{tikzpicture}
\caption{The geometry near a boundary point $x_0$.}
    \label{fig:geometry}
\end{figure}
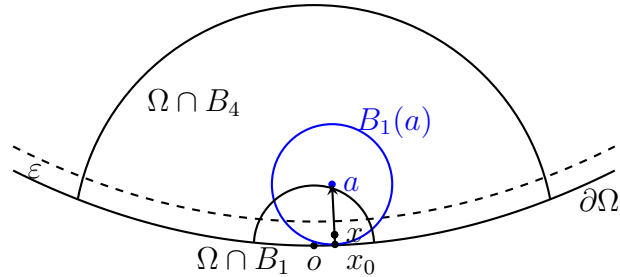
    
    We decompose the solution into two parts: 
    \[
    u =w + h \quad \text{in }B_1(a),
    \]
    with the potential part $w$ satisfying
    \begin{equation}
        \begin{cases}
            \flap w = g,&\text{in } B_1(a),\\
            w=0, &\text{in } B_1(a)^c,\\
        \end{cases}
    \end{equation}
    and the harmonic part $h$ solving
    \begin{equation}
        \begin{cases}
            \flap h = 0,&\text{in } B_1(a),\\
            h=u, &\text{in } B_1(a)^c.\\
        \end{cases}
    \end{equation}
    
    To estimate the potential part, we apply Theorem \ref{otonserra} to arrive at
    \begin{equation}\label{w}
        \norm{w} _{C^s(\R ^n)}\le \norm{g} _{L^\infty(B_1(a))}\le \norm{g} _{L^\infty(\Omega\cap B_4)}.
    \end{equation}
    
   For the harmonic part $h$, we first derive the following decay estimate\begin{equation}\label{decay estimate}
        |h(x)|\le C\{\norm{u}_{L^\infty(\Omega \cap B_4)}+\norm{u}_{\cL_{2s}}\}\operatorname{dist}(x,\partial \Omega)^{\alpha _0},\quad x\in (\Omega \backslash \Omega _\varepsilon) \cap B_1.
    \end{equation}
    where $\alpha _0 = \min(s,1-s)$, and $C$ is a constant depending on $n,s$, and $\norm{\partial \Omega}_{C^{1,1}_{loc}}$.

  In fact, by the Poisson representation, 
   \begin{align*}
    h(x) & = \int _{B_1(a)^c}\frac{(1-|x-a|^2)^s}{(|y-a|^2-1)^s}\frac{u(y)}{|x-y|^n}dy\\
    &\le Cd^s\int _{B_1(a)^c}\frac{u(y)}{(|y-a|^2-1)^s|x-y|^n}dy\\
    & = Cd^s\left(\int _{B_1(a)^c\backslash B_{1/2}(x_0)}+\int _{B_1(a)^c\cap B_{1/2}(x_0)} \right)\frac{u(y)}{(|y-a|^2-1)^s|x-y|^n}dy \\
    & :=J_1+J_2,
\end{align*}
with
\begin{align}
    J_1  & = Cd^s \int _{B_1(a)^c\backslash B_{1/2}(x_0)}\frac{u(y)}{(|y-a|^2-1)^s|x-y|^n}dy\nonumber\\
    &\le Cd^s\left(\int _{B_2(a)^c}\frac{u(y)}{(1+|y-a|)^{n+2s}}dy + \int _{B_2(a)\backslash(B_1(a)\cup B_{1/2}(x_0))}\frac{u(y)}{(|y-a|-1)^{2s}}dy\right)\nonumber \\
    & \le Cd^s\left(\norm{u}_{\cL_{2s}}+\norm{u}_{L^\infty(\Omega\cap B_4)}\right).\label{J1}
\end{align}

Now we analyze $J_2$. After a rigid transformation, we may assume $x_0=0$, $\nu = e_n$, $a = (0,1)$ and $x = (0,d)$.
Let $y_n=\varphi (y')$ denotes the graph of $\partial \Omega $ near $0$:
\[
\partial \Omega \cap B_4 : = \{(y', \varphi(y')): |y'|< 4\}.
\] 
Since $\partial \Omega\cap B_4 $ is $C^{1,1}$, there exists a constant $M$ depending on $\partial \Omega$, such that  $$|\varphi (y')|< M |y'|^2 $$ for any $|y'|<4$.
We calculate it  by using the sphere polar coordinates to first $n-1$ variables: 

\begin{align*}
    J_2 & = Cd^s \int _{B_1(a)^c\cap B_{1/2}\cap\Omega}\frac{u(y)}{(|y-a|^2-1)^s|x-y|^n}dy\\
    & = Cd^s \int _{|y'|<\frac{1}{2}}\int _{\varphi (y')}^{1-\sqrt{1-|y'|^2}}\frac{u(y)}{(|y'|^2+y_n(y_n-2))^s}\frac{1}{(|y'|^2+(d-y_n)^2)^{\frac{n}{2}}}dy_n\,dy'\\
    & \le Cd^s\int _0^{\frac{1}{2}}\int _{\mathbb{S}^{n-2}}\int _{-M r^2}^{1-\sqrt{1-r^2}}\frac{u(y)}{(r^2+y_n(y_n-2))^s}\frac{r^{n-2}}{(r^2+(d-y_n)^2)^{\frac{n}{2}}}dy_n\,d\sigma\,dr
\end{align*}

To continue, we change the variables $y_n =r^2z$:
\begin{align}
    J_2 & \le Cd^s\int _0^{\frac{1}{2}}\int _{\mathbb{S}^{n-2}}\int _{-M r^2}^{1-\sqrt{1-r^2}}\frac{u(y)}{(r^2+y_n(y_n-2))^s}\frac{r^{n-2}}{(r^2+(d-y_n)^2)^{\frac{n}{2}}}dy_n\,d\sigma\,dr\nonumber\\
    &  = Cd^s\int _0^{\frac{1}{2}} r^{-2s}\int _{\mathbb{S}^{n-2}}\int _{-M }^{\frac{1}{1+\sqrt{1-r^2}}}\frac{u(y)}{(1+z(r^2z-2))^s}\frac{1}{(1+(\frac{d}{r}-rz)^2)^{\frac{n}{2}}}dz\,d\sigma\,dr\nonumber\\
   & = Cd^s\left(\int _0^{d}+\int _{d}^{\frac{1}{2}} \right)r^{-2s}\int _{\mathbb{S}^{n-2}}\int _{-M}^{\frac{1}{1+\sqrt{1-r^2}}}\frac{u(y)}{(1+z(r^2z-2))^s}\frac{1}{(1+(\frac{d}{r}-rz)^2)^{\frac{n}{2}}}dz\,d\sigma\,dr\nonumber\\
   &:= J_{21}+J_{22}.\label{J2}
\end{align}

To calculate $J_{21}$, we notice that for $r\in (0,d)$,  ${(1+(\frac{d}{r}-rz)^2)^{-\frac{n}{2}}}$ behaves like $r^nd^{-n}$. 
\begin{align}
    J_{21}&\le Cd^{s-n}\int _0^dr^{n-2s}dr\int _{\mathbb{S}^{n-2}}\int _{-M}^{\frac{1}{1+\sqrt{1-r^2}}}\frac{u(y)}{(1+z(r^2z-2))^s}dz\,d\sigma\nonumber\\
    &\le Cd^{s-n}\norm{u}_{L^\infty(\Omega\cap B_4)}\int _0^dr^{n-2s}dr\int _{-M}^{\frac{1}{1+\sqrt{1-r^2}}}\frac{1}{(1+z(r^2z-2))^s}dz\\
    &\xlongequal{t=1+z(r^2z-2)}Cd^{s-n}\norm{u}_{L^\infty(\Omega\cap B_4)}\int _0^dr^{n-2s}dr\int _{0}^{C(M)}\frac{1}{t^s}\frac{1}{2\sqrt{1-r^2(1-t)}}dt\nonumber\\
    &\le Cd^{s-n}\norm{u}_{L^\infty(\Omega\cap B_4)}\int _0^dr^{n-2s}dr\int _{0}^{C(M)}\frac{1}{t^s}dt\nonumber\\
    &\le C(M)d^{1-s}\norm{u}_{L^\infty(\Omega\cap B_4)}.\label{J21}
\end{align}
Here we have used the fact that $1+z(r^2z-2)$ is decreasing when $z\in (-\infty, \frac{1}{1+\sqrt{1-r^2}})$, and the Jacobian
\[
\frac{dt}{dz} = \frac{1}{2\sqrt{1-r^2(1-t)}} 
\]
is bounded when $d$ is sufficiently small and $t\in (0,C(M))$, where $C(M)$ is a positive constant depending on $M$.

For $J_{22}$, since the term ${(1+(\frac{d}{r}-rz)^2)^{-\frac{n}{2}}}$ is bounded, then similar to the estimate of $J_{21}$, we have
\begin{align}
    J_{22}&\le Cd^s\int _d^{\frac{1}{2}}r^{-2s}dr\int _{\mathbb{S}^{n-2}}\int _{-M}^{\frac{1}{1+\sqrt{1-r^2}}}\frac{u(y)}{(1+z(r^2z-2))^s}dz\,d\sigma\nonumber\\
    &\le Cd^s\norm{u}_{L^\infty(\Omega\cap B_4)}\int _0^{C(M)}\frac{1}{t^s}dt\nonumber\\
    &\le C(M)(d^s+d^{1-s})\norm{u}_{L^\infty(\Omega\cap B_4)}.\label{J22}
\end{align}
Combining \eqref{J1}, \eqref{J2}, \eqref{J21} and \eqref{J22}, we obtain the decay estimate \eqref{decay estimate}.

\medskip

To continue estimate $\norm{u}_{C^{\alpha_0}(\Omega\cap B_{1/2})}$, we need the interior H\"older  estimate from \cite{chen2025refinedregularitynonlocalelliptic}.
\begin{proposition}[Interior H\"older regularity]\label{Holder thm}
    Assume $0<s<1$, and $u\in C^{1,1}_{loc}(\R^n)\cap \mathcal{L}_{2s}$ is a nonnegative solution of 
    \begin{equation}
        \flap u(x) = g(x),\quad x\in B_1
    \end{equation}
    where $g\in L^\infty(B_1)$. Then for any $0<\varepsilon <2s$,
    \begin{equation}
        \norm {u}_{C^{[2s-\varepsilon],\{2s-\varepsilon\}}(B_{1/2})}\le C(\norm{g}_{L^\infty(B_1)}+\norm{u}_{L^\infty(B_1)}).
    \end{equation}
\end{proposition}
Based on this, we arrive at
\begin{lemma}\label{global holder lemma}
    The decay estimate \eqref{decay estimate} and Proposition \ref{Holder thm} implies that 
    \begin{equation}
        \norm{u} _{C^{\alpha _0}(\Omega\cap B_{1/2})}\le C(\norm{u}_{L^\infty(\Omega\cap B_4)}+\norm{g}_{L^\infty(\Omega\cap B_4)})
    \end{equation}
    where $C$ depends on $s,n$, and $\norm{\partial \Omega}_{C^{1,1}_{loc}}$, $\alpha_0:=\min\{s,1-s\}\in (0,1/2)$. 
\end{lemma}

\begin{proof}[Proof of Lemma \ref{global holder lemma}]
It suffices to prove the result for the harmonic part: 
\begin{equation}\label{holder h}
    \norm{h} _{C^{\alpha _0}(\Omega\cap B_{1/2})}\le C\{\norm{u}_{L^\infty(\Omega\cap B_4)}+\norm{g}_{L^\infty(\Omega\cap B_4)}\}.
\end{equation}
This combined with  \eqref{w}, we derive the $C^{\alpha _0}$ estimate of $u$.  

Now we prove \eqref{holder h}. Using Proposition \ref{Holder thm} to $h$ and by a standard covering argument, we obtain 
\begin{equation}\label{2.15}
    \norm{h} _{C^{\alpha _0}(\Omega _{\varepsilon/2}\cap B_1)}\le C\norm{u}_{L^\infty(\Omega\cap B_4)},
\end{equation}
Here $C$ depends on $n,s$ and $\varepsilon$. By virtue of  \eqref{decay estimate} and \eqref{2.15}, 
\begin{equation}\label{sup h}
    \norm{h} _{L^\infty(\Omega \cap B_1)}\le C\{\norm{u}_{L^\infty(\Omega\cap B_4)}+\norm{u}_{\cL_{2s}}\}.
\end{equation}

First we claim that 
\begin{equation}\label{2.17}
    \sup _{x,y\in \Omega\cap B_{1/2}}\frac{h(x)-h(y)}{|x-y|^{\alpha_0}}\le C\{\norm{u}_{L^\infty(\Omega\cap B_4)}+\norm{u}_{\cL_{2s}}\},
\end{equation}
this together with \eqref{sup h} imply
\begin{equation}\label{a1}
    \norm{h} _{C^{\alpha _0}(\Omega\cap B_{1/2})}\le C\{\norm{u}_{L^\infty(\Omega\cap B_4)}+\norm{u}_{\cL_{2s}}\}.
\end{equation}
Then we recall the weak Harnack inequality in \cite{Rosoton2019Boundary} for the special case for the fractional Laplacian: 
\begin{proposition}\label{Harnack thm}
    Assume $0<s<1$ and $C_0>0$. Then there exists a positive constant $C=C(n,s)$, such that for every nonnegative classical solution of
    \[
    \flap u(x)\ge -C_0, \quad x\in B_{1},
    \]
    we have
    \begin{equation}
        \int _{\R^n}\frac{u(y)}{(1+|y|)^{n+2s}}\,dy\le C(\inf _{B_{1/2}}u + C_0).
    \end{equation}
\end{proposition}
Applying this proposition to $u$ in $B_1(a)$, we derive that
\begin{equation}\label{a2}
    \norm{u}_{\cL_{2s}}\le C\{\inf _{B_{1/2}(a)}u +\norm{g}_{L^\infty(B_1(a))}\}\le C\{\norm{u}_{L^\infty(\Omega\cap B_4)} +\norm{g}_{L^\infty(\Omega\cap B_4)} \}.
\end{equation}
Combining \eqref{a1} and \eqref{a2}, we establish \eqref{holder h} and thus complete the proof of this lemma.

It remains to verify \eqref{2.17}. For any given $x,y\in \Omega \cap B_{1/2}$, if $x,y\in \Omega _{\varepsilon/2}\cap B_{1/2}$, it is done by the interior H\"older estimate \eqref{2.15}. 
Now we denote 
\[
d_x :=\operatorname{dist}(x,\partial\Omega),\,\,d_y :=\operatorname{dist}(y,\partial\Omega).
\]
Without loss of generality, we assume that $x\in(\Omega\backslash\Omega _{\varepsilon/2})\cap B_{1/2}$ and 
\[
r := d_x\le d_y.
\] 
 Notice that $B_r(x)\subset (\Omega\backslash \Omega _\varepsilon)\cap B_1$ for sufficiently small $\varepsilon$.
 
 We denote $h_r(z):=h(rz+x)$ in $B_1$,  then 
\[
\flap h_r  = 0\quad\text{in }B_1.
\]
By Propostion \ref{Holder thm},
\[
[h_r]_{C^{\alpha_0}(B_{1/2})}\le C\norm{h_r}_{L^\infty(B_1)},
\]
where $C$ is independent of $x$. 
That is,
\begin{equation}\label{semi h}
    [h]_{C^{\alpha _0}(B_{r/2}(x))}\le Cr^{-\alpha_0}\norm{h}_{L^\infty(B_r(x))}\le C\{\norm{u}_{L^\infty(\Omega\cap B_4)}+\norm{u}_{\cL_{2s}}\}.
\end{equation}
Here we have used the decay estimate \eqref{decay estimate}.

Then \eqref{sup h} and \eqref{semi h} imply \eqref{2.17} as shown in the following:

If $|x-y|<r/2$, then by \eqref{semi h},
\[
\frac{|h(x)-h(y)|}{|x-y|^{\alpha_0}}\le [h]_{C^{\alpha _0}(B_{r/2}(x))}\le C\{\norm{u}_{L^\infty(\Omega\cap B_4)}+\norm{u}_{\cL_{2s}}\}.
\]

If $|x-y|\ge r/2$, then $d_y\le d_x+|x-y|\le 3|x-y|$, and hence
\begin{align*}
    \frac{|h(x)-h(y)|}{|x-y|^{\alpha_0}}&\le  C(d_x^{-\alpha _0}h(x) + d_y^{-\alpha _0}h(y))\le C\{\norm{u}_{L^\infty(\Omega\cap B_4)}+\norm{u}_{\cL_{2s}}\}.
\end{align*}
Here we have used \eqref{decay estimate}.

This validates \eqref{2.17} and complete the proof of the lemma.
\end{proof}
By virtue of Lemma \ref{global holder lemma}, if $s\le 1/2$, then $\alpha _0 = \min(s,1-s)=s$, and the proof of Theorem \ref{bdry C^s} is completed. 

Now suppose $s>1/2$. Starting from the  $C^{\alpha _0}$ boundary H\"older regularity of $u$, we apply successive iterations to 
increase the H\"older index until it reaches the desired index $s$. More precisely, we increase the expontent in the decay estimate \eqref{decay estimate} step by step
and after finite many steps, it will surpass $s$.
 Then   Lemma \ref{global holder lemma} with $\alpha_0$ replaced by $s$ implies the $C^{s}$ boundary regularity of $u$.

It suffices to reconsider $J_2$. Recalling that
\[
J_2  = Cd^s \int _{B^c\cap B_{1/2}(0)\cap\Omega}\frac{u(y)}{(|y-a|^2-1)^s|b^k-y|^n}dy.
\]
From now on, we denote 
\[
A := \norm{u}_{L^\infty(\Omega\cap B_4)}+\norm{u}_{\cL_{2s}}.
\]
For any $y\in \Omega \cap B_{1/2}$, by \eqref{decay estimate}, we have  $$u(y)\le CA|y|^{1-s}.$$ It follows that
\begin{align*}
    J_2  &= Cd^s \int _{B^c\cap B_{1/2}(0)\cap\Omega}\frac{u(y)}{(|y-a|^2-1)^s|b^k-y|^n}dy\\
    & \le CAd^s\int _0^{\frac{1}{2}}\int _{-M r^2}^{1-\sqrt{1-r^2}}\frac{(r^2+y_n^2)^{\frac{1-s}{2}}}{(r^2+y_n(y_n-2))^s}\frac{r^{n-2}}{(r^2+(d-y_n)^2)^{\frac{n}{2}}}dy_n\,dr\\
    &  = CAd^s\int _0^{\frac{1}{2}} r^{1-s}\cdot r^{-2s}\int _{-M }^{\frac{1}{1+\sqrt{1-r^2}}}\frac{(1+r^2z^2)^{\frac{1-s}{2}}}{(1+z(r^2z-2))^s}\frac{1}{(1+(\frac{d}{r}-rz)^2)^{\frac{n}{2}}}dz\,dr\\
     & \le CAd^s\int _0^{\frac{1}{2}} r^{1-3s}\int _{-M }^{\frac{1}{1+\sqrt{1-r^2}}}\frac{1}{(1+z(r^2z-2))^s}\frac{1}{(1+(\frac{d}{r}-rz)^2)^{\frac{n}{2}}}dz\,dr\\
   & = CAd^s\left(\int _0^{d}+\int _{d}^{\frac{1}{2}} \right)r^{1-3s}\int _{-M}^{\frac{1}{1+\sqrt{1-r^2}}}\frac{1}{(1+z(r^2z-2))^s}\frac{1}{(1+(\frac{d}{r}-rz)^2)^{\frac{n}{2}}}dz\,dr\\
   &:= J_{21}+J_{22}.
\end{align*}
Similarly, we have:
\begin{align*}
    J_{21} &\le Cd^{s-n}A\int _0^dr^{n+1-3s}dr\int _{0}^{C(M)}\frac{1}{t^s}dt\\
    & \le C(M)Ad^{2-2s}.
\end{align*}
\begin{align*}
    J_{22}&\le C\norm{u}_{L^\infty(\Omega\cap B_4)}d^s\int _d^{\frac{1}{2}}r^{1-3s}dr\int _0^{C(M)}\frac{1}{t^s}dt\\
    &\le CA(d^s+d^{2-2s}).
\end{align*}
Define $\alpha _1 =\min (s,2-2s)$.  Lemma \ref{global holder lemma} implies 
\[
\norm{u} _{C^{\alpha _1}(\Omega\cap B_{1/2})}\le C(\norm{u}_{L^\infty(\Omega\cap B_4)}+\norm{g}_{L^\infty(\Omega\cap B_4)}).
\]
If $s\le 2/3$, then $\alpha _1 = s$ and we are done. Otherwise we repeat the above procedure  to improve the H\"older index. Continuing this way and after $k$ steps, we  derive the $C^{\alpha _k}$ estimate of $u$, where
\[
\alpha  _k = \min\{s,k+1-(k+1)s\}.
\]
Since $s\in (0,1)$ is a fixed number, for sufficiently large $k$, we have $s\le \frac{k+1}{k+2}$, and consequently, we derive
\[
\norm{u} _{C^{s}(\Omega\cap B_{1/2})}\le C(\norm{u}_{L^\infty(\Omega\cap B_4)}+\norm{g}_{L^\infty(\Omega\cap B_4)}).
\]
This completes the proof of Theorem \ref{bdry C^s}.   
\end{proof}

\section{A Priori Estimates of Solutions}

The boundary estimates in the previous section enable us to apply the blow-up and rescaling argument to derive a priori estimates for nonnegative solutions to nonlinear fractional equations on unbounded domains with boundary. 

Consider the equation
\begin{equation}\label{AP1}
\begin{cases}
    \flap u (x) = f(x,u)  , & x\in \Omega,\\
    u (x)  = 0 , &x\in \Omega ^c.
\end{cases}
\end{equation} 
Here $\Omega \subset \R ^n$ is an unbounded domain with uniformly $C^{1,1}$ boundary, $f$ satisfies the following condition:
\begin{itemize}
    \item $f(x,t):\Omega \times [0,\infty)\to\R$ is uniformly H\"older continuous with respect to $x$ and continuous with respect to $t$.
\end{itemize}
\begin{theorem}\label{A1}
Assume $1<p<\frac{n+2s}{n-2s}$,  $f$ satisfies 
\begin{equation}\label{growth f}
    f(x,t)\le C_0(1+t^p)
\end{equation}
uniformly for all $x$ in $\Omega$, and 
\begin{equation}\label{lim f}
    \lim _{t\to\infty}\frac{f(x,t)}{t^p} = K(x),
\end{equation}
    where $K(x)\in (0,\infty)$ is uniformly continuous and $K(\infty):=\lim _{|x|\to\infty }K(x) \in (0,\infty)$. 
Then there exists a constant $C$, such that 
\begin{equation}\label{bd1}
u(x)\le C,\quad\forall \, x\in \Omega
\end{equation}
holds for all nonnegative solutions $u$ of \eqref{AP1}.
\end{theorem}

\subsection{Proof of Theorem \ref{AP1}}
We will prove Theorem \ref{A1} by a contradiction argument. Suppose \eqref{bd1} is violated, then there exist a sequence of solutions $u_k$ of \eqref{AP1} and a sequence of points  $\{x^k\}\subset \Omega$ such that
\begin{equation}\label{contradiction}
    u_k(x^k) >k .
\end{equation}
Different from the bounded domain cases, each $x^k$ may not be chosen to be the maximum point of $u_k$, hence $u_k(x^k)$ cannot be the upper bound of $u_k$ in $\Omega$. However, we can choose a point $a^k$ near $x^k$, such that $u_k(a^k)$ can control $u_k$ in a neighborhood of $a^k$, and this is sufficient  for us to carry out the blow-up and rescaling. 

For any fixed large $R$, let
\[
r_k=2Ru_k^{-\frac{p-1}{2s}}(x^k).
\]

Define a function
\[
S_k(x)  = u_k(x) (r_k- |x-x^k|)^{\frac{2s}{p-1}}, \quad x\in \overline{B_{r_k}(x^k)}.
\]
Note that $\overline{B_{r_k}(x^k)}\cap \Omega ^c$ may not be empty.
Since $S_k$ vanishes at $\partial B_{r_k}(x^k)$, we can find $a^k\in B_{r_k}(x^k)$ such that
\[
S_k(a^k)  =\max _{B_{r_k}(x^k)}S_k(x)
\]
That is,
\begin{equation}\label{db}
    u_k(x)\le u_k(a^k)\frac{(r_k- |a^k-x^k|)^{\frac{2s}{p-1}}}{(r_k- |x-x^k|)^{\frac{2s}{p-1}}},\quad x\in B_{r_k}(x^k)
\end{equation}
In particular, taking $x= x^k$ in \eqref{db} yields,
\begin{equation}\label{xk}
    u_k(x^k)\le u_k(a^k)\frac{(r_k- |a^k-x^k|)^{\frac{2s}{p-1}}}{r_k^{\frac{2s}{p-1}}},
\end{equation}
which implies that
\[
u_k(x^k)\le u_k(a^k).
\]
Hence $a^k\in B_{r_k}(x^k)\cap \Omega$.
Set
\begin{equation}
    \lambda _k  = (u_k(a^k))^{-\frac{p-1}{2s}}.
\end{equation}
It follows from the definition of $r_k$ and \eqref{xk} that 
\begin{equation}\label{2rlmbda}
    2R\lambda _k \le r_k - |a^k-x^k|.
\end{equation}
Then for any $x\in B_{R\lambda _k}(a^k)$,
\[
|x-x^k| \le |x-a^k|+|a^k-x^k|\le   R\lambda _k +  |a^k-x^k|<r_k.
\]
This implies
\[
B_{R\lambda _k}(a^k)\subset B_{r_k}(x^k).
\]

In addition, one can derive
\begin{equation}\label{db ieq}
     r_k-|a^k-x^k|\le 2(r_k-|x-x^k|),\quad x\in B_{R\lambda _k}(a^k).
\end{equation}
Indeed, for any $x\in B_{R\lambda _k}(a^k)$, using \eqref{2rlmbda},
\begin{align*}
    r_k-|a^k-x^k| & =  r_k-|a^k-x^k| +2R\lambda _k-2R\lambda _k\\
    & \le 2(r_k-|a^k-x^k|-R\lambda _k)\\
    &\le 2(r_k- (|a^k-x^k| + |x-a^k| ))\\
    & \le 2(r_k- |x-x^k|).
\end{align*}

The combination of \eqref{db} with \eqref{db ieq} yields
\begin{equation}
    u_k(x)\le u_k(a^k) \cdot 2^{\frac{2s}{p-1}},\quad x\in B_{R\lambda _k}(a^k).
\end{equation}

Now we rescale the solution as 
\[
v_k(x) = \frac{1}{u_k(a^k)}u_k(\lambda _kx+a^k), \quad x\in \R ^n.
\]
By \eqref{AP1}, it is easy to check that the nonnegative function $v_k$ satisfies
\begin{equation}
    \begin{cases}
        \flap v_k(x)   = F_k(x,v_k(x)) , & x\in B_R(0)\cap \Omega _k,\\
    v_k(x)\le 2 ^{\frac{2s}{p-1}}, & x\in B_R(0),\\
    v_k(0) = 1.
    \end{cases}
\end{equation}
Here 
\[
\Omega _k:= \{x\in \R ^n : \lambda _kx+a^k\in \Omega \},
\]
\begin{equation}
    F_k(x,v_k(x))  = \lambda _k ^{\frac{2sp}{p-1}}f(\lambda _kx+a^k, \lambda _k ^{-\frac{2s}{p-1}}v_k(x)).
\end{equation}
Then the growth condition \eqref{growth f} indicates
\begin{equation}\label{Fk bdd}
    F_k(x,v_k(x))\le C_0(1+v_k^p(x)),\quad \forall\, x\in\Omega _k.
\end{equation}
Now we denote 
\[
d_k =  \mathrm{dist}(a^k, \partial \Omega).
\]
Recalling $\lambda _k\to 0$ as $k\to\infty$, we can always find a subsequence of $\{\frac{d_k}{\lambda _k}\}$ (still denoted as $\{\frac{d_k}{\lambda _k}\}$), such that $\lim _{k\to\infty}\frac{d_k}{\lambda  _k}$ exists. According to the value of the limit, it can be divided into three cases: 0, $\infty$, and a positive number $H$.

We first claim that the case 
\begin{equation}\label{lim0}
    \lim _{k\to\infty}\frac{d_k}{\lambda  _k} = 0
\end{equation}
cannot happen. 

Suppose  \eqref{lim0} holds.
Since the distance $\mathrm{dist}( a^k, \partial \Omega ) = d_k\to 0$ as $k\to\infty$, for sufficiently large $k$, we can find a unique point $ p^k\in  \partial \Omega$, such that 
\[
| p^k- a^k| = \mathrm{dist}( a^k, \partial \Omega ).
\]
Suppose $\partial \Omega$ can be written as a graph of $C^{1,1}$ function near $p^k$ as follows: 

If $z\in\partial \Omega $ is near to $p^k$, then 
\begin{equation}
    z_n = \varphi_k(z') 
\end{equation}
for some $C^{1,1}$ function $\varphi _k$ satisfying 
\begin{equation}
    \norm{\varphi _k} _{C^{1,1}(B_1({p^k}'))}\le M
\end{equation}
for some positive constant $M$. 

Denote $q^k = \frac{p^k-a^k}{\lambda _k}$ Consider an isometry $T$ such that $T(q^k)= 0 $ and $\tilde a^k := T(0)= (0',d_k/\lambda_k)\in \R^{n-1}\times \R$, denote $\tilde \Omega _k = T(\Omega _k)$ with 
\[
\Omega _k:= \{y\in \R ^n : \lambda _ky+a^k\in \Omega \}
\]

Define $\tilde v_k : = v_k\circ T^{-1}$, and $\tilde F_k(x, \tilde v_k(x))=F_k(T^{-1}x, \tilde v_k(x))$.
\begin{equation}\label{tvk}
    \begin{cases}
        \flap \tilde v_k(x)  = \tilde F_k(x, \tilde v_k(x)), & x\in B_R(\tilde a^k)\cap \tilde\Omega _k,\\
     \tilde v_k(x)\le 2 ^{\frac{2s}{p-1}}, & x\in B_R(\tilde a^k),\\
   \tilde v_k(\tilde a^k) = 1.
    \end{cases}
\end{equation}
And $\partial \tilde\Omega _k$ can be written as, if $ x\in\partial \tilde\Omega _k \cap B_1$, then
\[
x_n = \psi _k (x') 
\]
with 
\[
\psi(0)=0, \quad \nabla\psi(0)= 0.
\]
 
A direct computation implies 
\begin{equation}
    \norm{\psi_k}_{C^{1,1}(B_1)}\le \lambda _k\norm{\varphi _k}_{C^{1,1}(B_{2\lambda _k}({p^k}'))}\le \lambda _kM.
\end{equation}

Roughly speaking, the boundary $\partial \tilde\Omega _k$ tends to become flatter as the blowing up process continues.
Especially, for sufficiently large $k$, the $C^{1,1}$ norm of $\partial \tilde\Omega _k$ near $B_1$ is uniformly bounded with respect to $k$.

We will prove that $\tilde v_k(\tilde a^k) \to 0$ as $k\to\infty$ by using the boundary H\"older estimate derived in Section 2, which contradicts the fact that $\tilde v_k(\tilde a^k) = 1$ for each $k$.

If we choose $R$ large enough such that $ B_1\subset B_R(\tilde a^k)$, then Theorem \ref{bdry C^s} for $\tilde v_k$ in $ \tilde\Omega_k\cap B_1$ implies
\begin{align*}
    \norm{\tilde v_k}_{C^s(B_{1/8}(0))}&\le C.
\end{align*}
Then
\[
\lim _{k\to\infty}| \tilde v_k(b^k)-\tilde v_k(0)| =0.
\]
Given that $\tilde{v}_k(\tilde a^k) \to 0$, the case where $d_k / \lambda_k \to 0$ cannot occur.

It remains two cases: $\lim _{k\to\infty}\frac{d_k}{\lambda _k} = \infty$, and $\lim _{k\to\infty}\frac{d_k}{\lambda _k}=H>0$.

Suppose
\[
\lim _{k\to\infty}\frac{d_k}{\lambda _k} = \infty.
\]
For any given $R$ and for sufficiently large $k$, $B_R(0)\subset  \Omega _k$, and $v_k$ satisfies 
\begin{equation}\label{vk}
    \begin{cases}
        \flap v_k(x)  = F_k(x,v_k(x)), & x\in B_R(0),\\
    v_k(x)\le 2 ^{\frac{2s}{p-1}}, & x\in B_R(0),\\
    v_k(0) = 1.
    \end{cases}
\end{equation}
Lemma \ref{Holder thm} implies that for all $0<\varepsilon<2s$,
\[
\norm{v_k}_{C^{[2s-\varepsilon],\{2s-\varepsilon\}}(B_{3R/4}(0))}\le C.
\]

Based on the H\"older continuity of the right hand side, in order to enhance the regularity to a higher order, we apply the following estimates from \cite{chen2025refinedregularitynonlocalelliptic}.

\begin{proposition}[Interior Schauder regularity]\label{Schauder thm}
    Assume $0<s<1$, and $u\in C^{1,1}_{loc}(\R^n)\cap \mathcal{L}_{2s}$ is a nonnegative solution of 
    \begin{equation}
        \flap u(x) = f(x),\quad x\in B_1
    \end{equation}
    where $f\in C^\alpha(B_1)$ for some $0<\alpha<1 $. Then for any $0<\varepsilon <2s$,
    \begin{equation}
        \norm {u}_{C^{[2s+\alpha-\varepsilon],\{2s+\alpha-\varepsilon\}}(B_{1/2})}\le C(\norm{f}_{C^\alpha(B_1)}+\norm{u}_{L^\infty(B_1)}).
    \end{equation}
\end{proposition}

It follows from this proposition that there exist some positive constants $\beta \in (0,1)$ and $C$ such that
\begin{equation}\label{vk schauder}
    \norm{v_k}_{C^{[2s+\beta],\{2s+\beta\}}(B_{R/2}(0))}\le C.
\end{equation}
We have already shown that given any $R>0$, there exists a large $K>0$, such that for any $k>K$, $v_k$ satisfies \eqref{vk schauder}. Applying the Arzel\'a-Ascoli theorem and a standard diagonal argument, we conclude that up to a subsequence, $\{v_k\}$ converges to some $v$ in $C^{[2s+\beta],\{2s+\beta\}}_{loc}(\R ^n)$.  \eqref{vk} implies that  $v\in C^{[2s+\beta],\{2s+\beta\}}_{loc}(\R ^n)$ is bounded in $\R ^n$, hence $v\in \mathcal{L}_{2s}(\R^n)$ and $\flap v$ is well-defined. Moreover, 
\begin{equation}
    v(0) = \lim _{k\to\infty}v_k(0) = 1.
\end{equation}
According to the equation, $\flap v_k$ converges pointwisely because the right hand side converges pointwisely. Therefore, a convergence result \cite{Du2023blowup} shows that: 
There exists a constant $b\ge 0$ such that 
\begin{equation}
    \lim _{k\to\infty}\flap v_k(x) = \flap v(x) - b,\quad \forall\, x\in\R ^n.
\end{equation}
On the other hand, recalling that
\[
F_k(x,v_k(x)) = \lambda _k ^{\frac{2sp}{p-1}}f(\lambda _kx+a^k, \lambda _k ^{-\frac{2s}{p-1}}v_k(x)).
\]
For any fixed $x$, recalling that $\lambda _k\to 0$ as $k\to\infty$. Therefore $\{\lambda_k x+a^k\}$ converges to a point $a$ (it may be $\infty$) which is independent of $x$.
Then in the sense of a subsequence,
\[
\lim _{k\to\infty}F_k(x,v_k(x)) = K(a)v^p(x).
\]
Due to the equation, 
\[
\lim _{k\to\infty}\flap v_k(x) = \lim _{k\to\infty} F_k(x,v_k(x))= K(a)v^p(x),\quad \forall\, x\in\R ^n.
\]
Therefore $v\in C^{[2s+\beta],\{2s+\beta\}}_{loc}(\R ^n)\cap \mathcal{L}_{2s}(\R^n)$ satisfies
\[
 \flap v(x)  = K(a)v^p(x) + b,\quad \forall\, x\in\R ^n.
\]
We now show that $b=0$. For any $R>0$, define
\[
v_R(x) = C_{n,s}\int _{\R^n}\frac{\chi _{B_R(0)}(y)(K(a)v^p(y) + b)}{|x-y|^{n-2s}} dy.
\]
then $v-v_R$ satisfies:
\begin{equation}
    \begin{cases}
        \flap(v-v_R)(x)\ge 0, \quad x\in \R^n,\\
        \liminf _{|x|\to\infty}(v-v_R)(x)\ge 0.
    \end{cases}
\end{equation}
By the maximum principle, we have $v-v_R\ge 0$ in $\R^n$. Let $R\to\infty$, we have
\begin{equation}
    v(x)\ge C_{n,s}\int _{\R^n}\frac{K(a)v^p(y) + b}{|x-y|^{n-2s}}dy.
\end{equation}
If $b>0$, the integral on the right hand side divergence, which contradicts the fact that $v(x)$ is finite. Hence $b=0$. Therefore the nonnegative function $v$ satisfies 
\begin{equation}
    \flap v(x) = K(a)v^p(x),\quad x\in\R^n.
\end{equation}
The Liouville theorem in \cite {Chen2017direct} indicates that $v\equiv 0$ if $1<p<\frac{n+2s}{n-2s}$, which contradicts  $v(0)=1$.

Suppose
\[
\lim _{k\to\infty}\frac{d_k}{\lambda _k} = H
\]
for some positive constant $H$. Since $\lambda _k\to 0$, the domain $ \Omega_k$ becomes the half space $$\R^n_{+H}:=\{x\in \R ^n: x_n>-H\}.$$ Unlike the previous case, no matter how large $k$ is, when $R$ is sufficiently large, $B_R(0)$ will not be contained in $\Omega _k$. 

We will prove that $\{v_k\}$ converges to a function $v$ in $C^{[2s+\beta],\{2s+\beta\}}_{loc}(\R^n_{+H})$ in the sense of subsequence. That is, for any compact set $E\subset \R^n_{+H}$, $v_k\to v$ in $C^{[2s+\beta],\{2s+\beta\}}(E)$. We use the similar argument as in \cite{Chen2016directblowup}. 

Fix any positive integer $m$, consider $R=2m$. Define
\[
D_m := \{x\in \overline{B_m(0)} : x_n\ge -H\frac{m-1}{m}\}.
\]
For sufficiently large $k$, we may assume 
\begin{itemize}
    \item $D_m\subset B_R\cap  \Omega_k$,
    \item $B_{\frac{H}{2m}}(x)\subset  \Omega_k$ for any $x\in D_m$,
\end{itemize}
because $ \tilde\Omega _k$ is almost flat. Then $v_k$ is uniformly bounded in each $B_{\frac{H}{2m}}(x)$.

Applying the interior H\"older and Schauder estimates in $B_{H/2m}(x)$,  similar to the bootstrap argument in the previous case, we obtain that there exist positive constants $\beta\in (0,1)$ and $C$ such that
\begin{equation}
    \norm{v_k} _{C^{[2s+\beta],\{2s+\beta\}}\left(B_{H/4m}(x) \right)}\le C.
\end{equation}
Then 
\begin{align}
    \norm {v_k} _{C^{[2s+\beta],\{2s+\beta\}}(D_m)}&\le \sup _{x\in D_m}\norm{v_k}_{C^{[2s+\beta],\{2s+\beta\}}\left(B_{H/4m}(x) \right)} + C(H)m^{\{2s+\beta\}}\sup _{x\in D_m}\norm{v_k}_{C^{[2s+\beta]}\left(B_{H/4m}(x)\right)}\notag\\
    & \le C(H,m). \label{aa}
\end{align}
The Arzel\'a-Ascoli theorem and a standard diagonal argument implies that up to a subsequence, $\{v_k\}$ converges to some $v$ in $C^{[2s+\beta],\{2s+\beta\}}(D_m)$
for each positive integer $m$. Therefore $\{v_k\}$ converge to $v$ in $C^{[2s+\beta],\{2s+\beta\}}_{loc}(\R^n_{+H})$.

Unlike in the previous case, we cannot apply Du-Jin-Xiong-Yang's convergence result directly because their results require the convergence in $C^{[2s+\beta],\{2s+\beta\}}_{loc}(\R^n)$. We refine their results in the Appendix to the case of convergence in any unbounded domain $\Omega$ with boundary. Then similarly we have
\[
\begin{cases}
    \flap v(x) = K(a)v^p(x)+b,& x\in \R^n_+,\\
    v(x) = 0,& x\in  {\R^n_-},
\end{cases}
\]
here we have made a translation in $x_n$ direction. 

To show $b=0$, we introduce 
\[
v_R(x) = \int _{B_R(P_R)}G_R(x,y)(K(a)v^p(y)+b)dy,
\]
Here $G_R$ is the Green function in $B_R(P_R)$ with $P_R = (0,\cdots,0, R)$. The argument is similar: $v-v_R$ satisfies
\[
\begin{cases}
    \flap(v-v_R)\ge 0, & x\in B_R(P_R),\\
    v-v_R\ge0,& x\in \R^n\backslash B_R(P_R).
\end{cases}
\]
The maximum principle implies $v\ge v_R$ in $B_R(P_R)$. Let $R\to\infty$ we have
\begin{equation}
    v(x)\ge \int _{\R^n_+}G_\infty(x,y)(v^p(y)+b)dy,
\end{equation}
where $G_\infty(x,y)$ is the Green function in $\R^n_+$:
\[
G_\infty(x,y) = \frac{C_{n,s}}{|x-y|^{n-2s}}\int _{0}^{\frac{4x_ny_n}{|x-y|^2}}\frac{b^{s-1}}{(b+1)^{\frac{n}{2}}}db.
\]
We have used the fact that 
$G_R(x,y)\to G_\infty(x,y)$ as $R\to\infty$, and we refer to\cite[Lemma 3.2]{Fall2016monotonicity} for more details. If $b> 0$, 
\[
v(x) \ge b\int _{\R^n_+} G_\infty(x,y)dy = \infty,
\]
which contradicts the fact that $v(x)$ is finite. Therefore $b=0$ and $v$ satisfies
\begin{equation}
    \begin{cases}
    \flap v(x) = K(a)v^p(x),& x\in \R^n_+,\\
    v(x) = 0,& x\in  {\R^n_-},
\end{cases}
\end{equation}
The Liouville theorem in half space \cite{Chen2015Liouville,Fall2016monotonicity} implies that if $1<p<\frac{n+2s}{n-2s}$, then $v\equiv0$. This contradicts the fact that $v(0)=1$.

In summary, assuming that the a priori estimate \eqref{bd1} does not hold inevitably leads to a contradiction. Therefore we have proved its validity and complete the proof.

\section{Appendix}
Here is a generalization  of \cite[Theorem 1.1]{Du2023blowup}Du-Jin-Xiong-Yang's convergence result.
\begin{theorem}
    Assume $n\ge 1$, $s\in (0,1)$, $\beta\in (0,1)$. $\Omega$ is an unbounded domain with boundary. Suppose nonnegative functions $\{u_i\}\subset \mathcal{L}_{2s}\cap C^{2s+\beta}_{loc}(\Omega)$ vanishes outside $\Omega$, and $u\in \mathcal{L}_{2s}$ vanishes outside $\Omega$. If $\{u_i\}$
    converges to a function $u\in \mathcal{L}_{2s}$ in $C^{2s+\beta}_{loc}(\Omega)$, and $\{\flap u_i \}$ converges pointwisely in $\Omega$, then there exists a constant $b\ge 0 $ such that
    \begin{equation}
        \lim _{i\to\infty}\flap u_i(x) = \flap u(x) -b.
    \end{equation}
\end{theorem}
\begin{proof}
    Fix any $x\in\Omega$, consider $R\gg|x|+1$.
    \begin{align*}
        \flap u(x)-\flap u_i(x) & = c_{n,s}\int _{B_R}\frac{(u-u_i)(x)-(u-u_i)(y)}{|x-y|^{n+2s}}dy \\ &+c_{n,s}\int _{B_R^c}\frac{(u-u_i)(x)-u(y)}{|x-y|^{n+2s}}dy\\
        & + c_{n,s}\int _{B_R^c}\frac{u_i(y)}{|x-y|^{n+2s}}dy\\
        & := A_i(x,R) + E_i(x,R) + F_i(x,R).
    \end{align*}
    For $E_i$, since $u\in \mathcal{L}_{2s}$, 
    \begin{equation}
        \lim _{R\to\infty}\lim _{i\to\infty}E_i(x,R) = -c_{n,s}\lim _{R\to\infty}\int _{B_R^c}\frac{u(y)}{|x-y|^{n+2s}}dy =0.
    \end{equation}
    The consideration of $A_i$ is a little bit different with that in \cite{Du2023blowup}. In our cases, $B_R$ is not contained in $\Omega$ for sufficiently large $R$. Fix $\varepsilon >0$ such that $B_\varepsilon(x)\subset \Omega$. Then
    \begin{align*}
        A_i(x,R) & = c_{n,s}\int _{B_R\backslash B_\varepsilon(x)}\frac{(u-u_i)(x)-(u-u_i)(y)}{|x-y|^{n+2s}}dy \\
        &+c_{n,s}\int _{B_\varepsilon(x)}\frac{(u-u_i)(x)-(u-u_i)(y)}{|x-y|^{n+2s}}dy
    \end{align*}
    By dominated convergence theorem, the first term converges to 0 as $i\to\infty$. For the second term,
    \[
    \lim _{i\to\infty} c_{n,s}\int _{B_\varepsilon(x)}\frac{(u-u_i)(x)-(u-u_i)(y)}{|x-y|^{n+2s}}dy\le c_{n,s}\lim _{i\to\infty}[u-u_i]_{C^{2s+\beta}(B_\varepsilon(x))}\varepsilon ^\beta = 0.
    \]
    Hence $\lim _{i\to\infty}A_i(x,R)=0$. Therefore 
    \begin{equation}
        \lim _{R\to\infty}\lim _{i\to\infty}A_i(x,R)=0.
    \end{equation}
    Then the same argument in \cite{Du2023blowup} implies $\lim _{i\to\infty}F_i(x,R)$ exists and
    \[
    \lim _{R\to\infty}\lim _{i\to\infty}F_i(x,R) =\lim _{R\to\infty}\lim _{i\to\infty}F_i(0,R) = :b\ge 0.
    \]
\end{proof}

{\bf{Acknowledgements.}} 
 Guo is partially supported by the National Natural Science Foundation of China (Grant No. 12501145), the Natural Science Foundation of Shanghai (No. 25ZR1402207),   the China Postdoctoral Science Foundation (No. 2025T180838 and No. 2025M773061), the Postdoctoral Fellowship Program of CPSF (No. GZC20252004), and the Institute of Modern Analysis-A Frontier Research Center of Shanghai.
Li and Ouyang are partially supported by the National Natural Science Foundation of China (Grant No. W2531006, 12250710674 and 12031012) and the Institute of Modern Analysis-A Frontier Research Center of Shanghai.
 \medskip

{\bf{Date availability statement:}} Data will be made available on reasonable request.
\medskip

{\bf{Conflict of interest statement:}} There is no conflict of interest.

\bibliography{ref.bib}

\bigskip

Yahong Guo

School of Mathematical Sciences

Shanghai Jiao Tong University

Shanghai, 200240, P.R. China

yhguo@sjtu.edu.cn
\medskip

Congming Li

School of Mathematical Sciences

Shanghai Jiao Tong University

Shanghai, 200240, P.R. China

congming.li@sjtu.edu.cn
\medskip

Yugao Ouyang

School of Mathematical Sciences

Shanghai Jiao Tong University

Shanghai, 200240, P.R. China

ouyang1929@sjtu.edu.cn

\end{document}